\documentclass[a4paper,10pt]{article}

\usepackage[utf8]{inputenc}

\usepackage{amssymb}
\usepackage{amsthm}
\usepackage{verbatim}
\usepackage{graphicx}
\usepackage{epstopdf}
\usepackage{indentfirst}                     
\usepackage[english]{babel}
\usepackage{subfig}                          
\usepackage{caption}                         

\usepackage{chngpage, calc}                  

\usepackage[font=small]{quoting}             
\usepackage[autostyle]{csquotes}             

\usepackage[style=numeric,hyperref,backend=bibtexu, maxcitenames=5, maxbibnames=9]{biblatex}
\bibliography{bibliography}

\usepackage[bookmarks=false]{hyperref}                        
\usepackage{cleveref}               
\usepackage{amsmath}
\usepackage{bookmark}
\usepackage{mdframed}
\usepackage{xcolor}
\usepackage{enumerate}
\usepackage{algorithm}
\usepackage{algorithmic}

\usepackage{geometry}
\numberwithin{equation}{section}

\newcommand{\Tau}{\mathcal{T}}

\definecolor{lightblue}{rgb}{0.68, 0.85, 0.9}
\definecolor{pastelblue}{rgb}{0.68, 0.78, 0.81}
\definecolor{palegreen}{rgb}{0.6, 0.98, 0.6}
\definecolor{pastelred}{rgb}{1.0, 0.41, 0.38}

\mdfdefinestyle{exampledefault}{
rightline=false,bottomline=false,
topline=false,leftline=false,
backgroundcolor=white
}
\mdfdefinestyle{problemdefault}{
rightline=false,bottomline=false,
topline=false,leftline=false,
backgroundcolor=white
}
\mdfdefinestyle{badthingdefault}{
rightline=false,bottomline=false,
topline=false,leftline=false,
backgroundcolor=white
}

\newcommand{\omegae}{\omega_{\varepsilon}}
\newcommand{\Omegae}{\Omega_{\varepsilon}}
\newcommand{\ue}{u_{\varepsilon}}

\newcommand{\omegaez}{\omega_{\varepsilon}} 
\newcommand{\uez}{u_{\varepsilon}} 
\newcommand{\wez}{w_{\varepsilon}} 

\newcommand{\dOO}{\partial \Omega}

\newcommand {\half}{ \frac{1}{2} }
\newcommand {\normb}[1]{ \norm{\partial \Omega}{#1} }

\newcommand{\R}{\mathbb{R}}

\newcommand{\abs}[1]{\left| #1 \right|}

\newcommand{\norm}[2]{{\left\| #2 \right\|}_{#1}}

\newcommand{\diverg}{\mathrm{div}}

\newcommand{\normal}{\nu}

\theoremstyle{definition} 

\newtheorem{remark}{Remark}[section]

\theoremstyle{plain}
\newtheorem{theorem}{Theorem}[section]
\newtheorem*{theorem*}{Theorem}
\newtheorem{proposition}{Proposition}[section]
\newtheorem*{proposition*}{Proposition}

\renewbibmacro{in:}{}
\DeclareFieldFormat[article]{citetitle}{#1}
\DeclareFieldFormat[article]{title}{#1}
\DeclareFieldFormat[article]{pages}{#1}

\DeclareFieldFormat[inproceedings]{citetitle}{#1}
\DeclareFieldFormat[inproceedings]{title}{#1}

\hypersetup{
   colorlinks=true,
	linkcolor = {black},
	citecolor = {black},
}

\colorlet{linkequation}{blue}

\newcommand*{\SavedEqref}{}
\let\SavedEqref\eqref
\renewcommand*{\eqref}[1]{%
  \begingroup
    \hypersetup{
      linkcolor=linkequation,
      linkbordercolor=linkequation,
    }%
    \SavedEqref{#1}%
  \endgroup
}

\graphicspath{{img/}} 

\title{A reconstruction algorithm based on topological gradient \\ for an inverse problem related to a \\ semilinear elliptic boundary value problem}
\author{Elena Beretta\footnotemark[1] \footnotemark[3] \and  Andrea Manzoni\footnotemark[2] \and Luca Ratti\footnotemark[1]}
\date{ }

\begin{document}

\maketitle

\renewcommand{\thefootnote}{\fnsymbol{footnote}}
\footnotetext[1]{Dipartimento di Matematica, Politecnico di Milano, Piazza Leonardo da Vinci 32, I-20133 Milano, Italy, \texttt{elena.beretta@polimi.it}, \texttt{luca.ratti@polimi.it}}
\footnotetext[2]{CMCS-MATHICSE-SB, Ecole Polytechnique F\'ed\'erale de Lausanne, 
              Station 8, CH-1015 Lausanne, Switzerland, 
              \texttt{andrea.manzoni@epfl.ch}}
\footnotetext[3]{Corresponding author}

\renewcommand{\thefootnote}{\arabic{footnote}}


\begin{abstract}
	In this paper we develop a reconstruction algorithm for the solution of an inverse boundary value problem dealing with a semilinear elliptic partial differential equation of interest in cardiac electrophysiology.
The goal is the detection of small inhomogeneities located inside a domain $\Omega$, where the coefficients of the equation are altered, starting from observations of the solution of the equation on the boundary $\partial \Omega$. Exploiting theoretical results recently achieved in \cite{art:bcmp}, we implement a reconstruction procedure based on the computation of the topological gradient of a suitable cost functional. Numerical results obtained for several test cases finally assess the feasibility and the accuracy of the proposed technique.
\end{abstract}

\section{Introduction}
\label{sec:intro}

Consider the following Neumann problem, defined over $\Omega \subset \R^2$:
\begin{equation}\label{eq:prob}
	\left\{
	\begin{aligned}
		-\diverg(k(x) \nabla u) + \chi_{\Omega \setminus \omega}u^3 &= f \qquad \text{in } \Omega	\\
		\partial_{\normal} u &= 0 \qquad \text{on } \partial \Omega,
	\end{aligned}
	\right.
\end{equation}
where
\[
	k(x) = 
	\left\{ 
	\begin{aligned} 
		k_{in} & \text{ if } x \in \omega \\ 
		k_{out} & \text{ if } x \in \Omega \setminus \omega, 
	\end{aligned} 
	\right. 
	\quad k_{in} \neq k_{out},
\]
being $ 0 < k_{in} \ll k_{out}$ two positive scalars.
The boundary value problem \eqref{eq:prob} consists of a semilinear diffusion-reaction equation with discontinuous coefficients across the interface of an inclusion $\omega \subset \Omega$, in which the conducting properties are different from the background medium. Our goal is the detection of the inclusion from the knowledge of the value of $u$ on the boundary $\partial \Omega$, i.e., given the measured data $u_{meas}$ on the boundary $\partial \Omega$, to find $\omega \subset \Omega$ such that the corresponding solution $u$ of \eqref{eq:prob} satisfies
	\begin{equation} 
		u|_{\partial \Omega} = u_{meas}. 
		\label{eq:inv}
	\end{equation}
	Since at the state of the art very few works tackle similar inverse problems in a nonlinear context, the reconstruction problem to which this work is devoted is particularly interesting from both an analytic and a numerical standpoint. 
	
	The direct problem can be related to a meaningful application arising in cardiac electrophysiology, up to several, substantial simplifications. In that context (see \cite{book:sundes-lines}, \cite{book:pavarino}), the solution $u$ represents the electric transmembrane potential in the heart tissue, the coefficient $k$ is the tissue conductivity and the nonlinear reaction term encodes a ionic transmembrane current. A small inclusion $\omega$ models the presence of an early-stage ischemia, which causes a substantial alteration in the conductivity properties of the tissue.
 The long-term objective of our work is the identification of early-stage ischemic regions through a set of measurements of the electric potential acquired on the surface of the myocardium. Indeed, a map of the potential on the boundary of internal heart cavities can be acquired by means of non-contact electrodes carried by a catheter inside a heart cavity; this is the procedure of the so-called intracardiac electrogram technique, which has become a possible (but invasive) inspection technique for patients showing symptoms of heart failure. In order to tackle the reconstruction problem of early-stage ischemias from data acquired on the external surface of the heart with non-invasive technique (such as the electrocardiogram), one should consider the (more involved) coupled system composed by the heart and the torso: however, this is beyond the purposes of the present work.
We remark that our model is a simplified version of the more complex \textit{monodomain} model (see e.g. \cite{phd:tung}, \cite{book:sundes-lines}). The monodomain is a continuum model which describes the evolution of the transmembrane potential on the heart tissue according to the conservation law for currents and to a satisfying description of the ionic current, which entails the coupling with a system of ordinary differential equations for the concentration of chemical species. The ionic flows through the cellular membrane is indeed the driving mechanism of the electrical potential on the smallest scale. In this preliminary setting, we remove the coupling with the ionic model, adopt instead a phenomenological description of the ionic current, through the introduction of a cubic reaction term. Morover, we consider the stationary case in presence of a source term which plays the role of the electrical stimulus. Hence, the present analysis can be considered as a blueprint, useful to tackle the main mathematical challenges of the original problem and to set the starting point for a forthcoming research.
\par 
The well-posedness analysis of the proposed reconstruction problem presents severe mathematical difficulties. The linear counterpart of the problem, obtained when the nonlinear reaction term is removed, is strictly related to the \textit{inverse conductivity problem}, also called \textit{Calder\'{o}n problem}, which has been object of several studies in the last decades. Without additional hypotheses on the geometry of the inclusion, but only assuming a sufficient degree of regularity of the interface, uniqueness from knowledge of infinitely many measurements has been proved in \cite{art:isakov_discontinuous} and logarithmic-type stability estimates have been derived in \cite{art:ales}. Finitely many measurements are sufficient to determine uniquenely and in a stable (Lipschitz) way the inclusion introducing additional information either on the shape of the inclusion or on its size, e.g. when the inclusion belongs to a specific class of domains with prescribed shape, such as discs, polygons, spheres, cylinders, polyhedra (see \cite{art:isa-po}, \cite{book:ammari-kang}, \cite{art:hett-run}, \cite{art:barcelo}) or when the volume of the inclusion is small compared to the volume of the domain (see \cite{art:fried-vog}, \cite{art:cfmv}). This latter case is of particular interest for our purposes: indeed, we aim to reconstruct the position of the inhomogeneity with a single measurement, and the hypothesis of small dimension entails that we are looking for early-stage ischemias. Nevertheless, similar theoretical results have not been extended yet in our nonlinear case. At the state of the art, the most important theoretical result for the inverse problem in the nonlinear setting is an asymptotic expansion of the boundary potential with respect to the presence of an inclusion of small dimension, recently derived in \cite{art:bcmp}. 
\par
Concerning the numerical solution of the inverse reconstruction problem, different algorithms have been developed for the inverse (linear) conductivity problem in the case of small inclusions (see e.g. \cite{book:ammari-kang}, chapter 5), exploiting an asymptotic formula for the perturbation of the boundary potential to identify the location and additional features of the shape of the inclusions. We recall, for instance, the constant current projection algorithm in \cite{art:amm-seo}, the least-squares algorithm proposed in \cite{art:cfmv}, and the linear sampling method developed in \cite{art:bhv} for similar problems. Although these algorithms have proved to be effective, they strictly depend on the linearity of the problem, especially concerning the explicit formula for the Neumann function of the operator (involving single or double layer potentials), and the analytic expression of some particular solutions. For nonlinear problem at hand, similar techniques have been proposed only in \cite{art:bcmp}. Unlike this approach, we propose in this paper a topological optimization framework for the reconstruction of the center of the inclusion, by evaluating the topological gradient of a suitable cost functional with respect to the introduction of small inclusions within the domain. Numerical procedures based on topological optimization have been widely developed for linear problems in several contexts (see for instance \cite{art:ahm} and \cite{art:dro} for crack detection, \cite{art:ammgj} and \cite{book:rapun} for the detection of sound obstacles, \cite{art:larnier} and \cite{art:bergramusch} for image processing), and have been successfully applied for the inverse (linear) conductivity problem (\cite{art:cmm}, \cite{art:ams1}) to identify the position of the center of small conductivity inclusions. Moreover, a similar technique can be combined with an iterative algorithm such as the level set method or with the solution of a successive shape optimization problem, in order to achieve a full reconstruction both of the dimension and the shape of the inclusion (see \cite{art:abfkl} for the computation of the shape gradient in a linear context and \cite{art:cmm} for the interplay between topological and shape optimization; some applications of the level-set method are instead reported in \cite{art:santosa}, \cite{art:chan-tai}). Concerning nonlinear problems, we can recall some techniques related to sensitivity analysis for semilinear elliptic problems in \cite{art:iguernane}, \cite{art:scheid}, \cite{art:amstutzNL}, although in different contexts with respect to our application. We remark that the level-set method has been implemented for the reconstruction of extended inclusion in the nonlinear problem we are dealing with (see \cite{art:lyka2}, \cite{art:lyka}), but in a fully parametrized context, i.e. by evaluating the sensitivity of the cost functional with respect to a selected set of parameters related to the shape of the inclusion, treated as design variables.
\par
The structure of the paper is as follows. We recall in section \ref{sec:dir} a well-posedness result for the direct problem, together with the already mentioned asymptotic formula (proved in \cite{art:bcmp}), and show an additional estimate useful for the sake of reconstruction. Then, in section \ref{sec:rep} we set our problem in a topological optimization frame, introducing a cost functional and providing a representation formula for the topological gradient, taking advantage of a suitable adjoint problem. In section \ref{sec:algo} we exploit these results in order to set up a recostruction procedure which allows to detect the position of an inclusion of small dimension taking advantage of multiple measurements or of a single measurement on different regions of the boundary, respectively. In section \ref{sec:results} we report some numerical results obtained applying the algorithm in different test cases, in order to analyse the reconstruction of the inclusions and to test the stability of the procedure with respect to statistical noise on the boundary datum.

\section{Preliminary results on the direct problem}
\label{sec:dir}

As discussed in the previous section, the nonlinearity of the direct problem \eqref{eq:prob} yields a remarkable difference with respect to the majority of the works in literature involving reconstruction problems. Hence, it is important to recall some preliminary results, such as a well-posedness theorem for the direct problem and the asymptotic expansion of the boundary datum with respect to the introduction of a small inhomogeneity, on which we rely in order the develop our procedure. For the same purpose, we prove an estimates of the $L^2(\partial \Omega)$-norm of the perturbation.
\par
The weak formulation of the Neumann homogeneous problem \eqref{eq:prob} reads as follows: find $u \in V = H^1(\Omega)$ s.t.
\begin{equation}
	<T(u)-F,v>_{V^*,V} = 0 	\quad \forall v \in V,
\label{eq:weak}
\end{equation}
where $<\cdot,\cdot>_{V^*,V}$ is the duality pairing between $V$ and its dual space $V^*$ and $F, T(u) \in V^*$ are defined by:
\begin{equation}
\begin{aligned}
	<T(u),v>_{V^*,V} &= \int_{\Omega}{k(x) \nabla u \cdot \nabla v} + \int_{\Omega}{\chi_{\Omega \setminus \omega} u^3v}, \\
	<F,v>_{V^*,V} &= \int_{\Omega}{fv}, \qquad f \in L^p(\Omega), \quad p \geq 2.
\end{aligned}
\label{eq:weak2}
\end{equation}
\par
We will refer to those inclusions $\omegae$ of small dimensions which are well separated from the boundary, i.e. such that: 
\begin{equation}
	|\omegae| \rightarrow 0 \text{ as } \varepsilon \rightarrow 0
\label{eq:small}
\end{equation}
\vspace{-0.5cm}
\begin{equation}
	\exists K_0 \subset \Omega \textit{ compact s.t. } \omegae \subset \subset K_0, \quad dist(\partial \Omega, K_0) \geq d_0 > 0.
\label{eq:sep}
\end{equation}
Under these assumptions, the direct problem \eqref{eq:prob} is well-posed (see \cite{art:bcmp} for the proof):
\begin{proposition}
	For every forcing term $f \in L^p(\Omega)$, $p \geq 2$, and every admissible inclusion $\omegae$, problem \eqref{eq:weak} admits an unique solution $\ue \in V$.
	\label{prop:wp}
\end{proposition}
\par
Hereon we set $k_{out} = 1$ and $k_{in} = k \ll 1$, for the sake of simplicity.
In order to obtain an asymptotic expansion of the perturbation of the solution on the boundary resulting by the introduction of an inclusion of small dimensions, we introduce the \textit{unperturbed potential} $U$, which solves the problem \eqref{eq:prob} without any inclusion, that is:
\begin{equation}
\left\{
	\begin{aligned}
		-\Delta U + U^3 &= f \qquad &\text{in } \Omega	\\
		\partial_{\normal} U &= 0 \qquad &\text{on } \partial \Omega .
	\end{aligned}
	\right.
\label{eq:unperturbed}
\end{equation}
We also introduce the Neumann function $N_U(\cdot,y)$ related to the operator $-\Delta + 3 U^2$, i.e. the solution, for each $y \in \Omega$, of 
\begin{equation}
\left\{
	\begin{aligned}
		-\Delta N_U(x,y) + 3 U^2(x) N_U(x,y) &= \delta(x-y) \qquad &\text{in } \Omega	\\
		\partial_{\normal_x} N_U(x,y) &= 0 \qquad &\text{on } \partial \Omega .
	\end{aligned}
	\right. 
\label{eq:Green}
\end{equation}
\par
Moreover, we consider some differences with respect to the general context in which the asymptotic formula has been deduced in \cite{art:bcmp}. First of all, we restrict ourselves to a specific class of inclusions, namely those of the form: \vspace{-0.1cm}
\begin{equation}
	\omegae = (z + \varepsilon D) = \{x \in \Omega \ s.t.\ \exists d \in D: \ x = z+ \varepsilon d\} \vspace{-0.1cm}
\label{eq:incl2}
\end{equation}
being $D$ an open, bounded and regular domain containing the origin. The inclusion $\omegae$ therefore consists of a single connected set, with center $z$, fixed shape $D$ and small dimension. 
\par
Furthermore, we weaken the hypothesis on the source term reported in \cite{art:bcmp}, in view of the requirements of our reconstruction algorithm. Instead of requiring that $f$ is bounded from below from a positive constant (that is, $\exists m > 0$ s.t. $f(x) \geq m$ $a.e.$ in $\Omega$), we assume that $f$ does not identically vanish in $\Omega$. This does not affect the proof given in the original work: see the Appendix A for further details.
Then, we can formulate the result yielding the asymptotic expansion as follows:
\begin{theorem}
Let $\omegaez$ be a family of subdomains satisfying \eqref{eq:sep} and \eqref{eq:incl2}, $f \in L^p(\Omega)$, $p \geq 2$, $f \neq 0$, $\uez$ and $U$ the solution to \eqref{eq:prob} and \eqref{eq:unperturbed} respectively. Then, there exists a symmetric matrix $\mathcal{M} \in \R^{2\times 2}$ s.t. $\wez = \uez - U$ satisfies, for any $y \in \partial \Omega$: 
\begin{equation}
\begin{aligned}
\wez(y) = \varepsilon^d & \left[ (1-k)\nabla U(z)^T \mathcal{M} \nabla_x N_U(z,y) + U^3(z)N_U(z_i,y)\right]  + o(\varepsilon^d), \quad \textit{ as } \varepsilon \rightarrow 0.
\end{aligned}
\label{eq:exp2}
\end{equation}
\label{th:sviluppo}
\end{theorem}
\begin{remark}
The matrix $\mathcal{M} \in \R^{2\times 2}$ appearing in \eqref{eq:exp2} is called \textit{polarization tensor} and depends only on the coefficient $k$ and on the shape $D$ of the inclusion. Moreover, it can be explicitely computed for some specific shapes (see, e.g., \cite{book:ammari-kang} for a detailed derivation): for instance, if the inclusion has circular shape, the following expression holds:
\begin{equation}
	\mathcal{M}= \frac{2}{1+k} \abs{D} \textbf{I}_{2 \times 2}.
\label{eq:circle}
\end{equation}
If the inclusion has elliptical shape with major axis aligned in the direction $\nu$ and ratio $r$ between the axes,
\begin{equation}
	\mathcal{M} = \mathcal{M}(k,\nu,r) = R^T \widetilde{\mathcal{M}} R,
\label{eq:ellipse}
\end{equation}
being
\[
	\widetilde{\mathcal{M}} = (k-1)\abs{D} \left( \begin{array}{cc} \frac{1+r}{1+kr} & 0 \\  0 & \frac{1+r}{r+k} \end{array} \right), \quad R = \left( \begin{array}{cc} \nu_x & -\nu_y \\  \nu_y & \nu_x \end{array} \right).
\]
\end{remark}
\par
We can exploit the expansion in \eqref{eq:exp2} to prove an additional result, useful in the sequel to get an asymptotic formula for the topological gradient of the cost functional we are going to introduce.
\begin{proposition}
	In the same hypotheses of Theorem \ref{th:sviluppo}, there exists a positive constant $C = C(k, d_0,f,\Omega)$ such that the perturbation on the boundary data $(\uez- U)|_{\partial \Omega}$ fulfills:
	\begin{equation}
		\norm{L^2(\partial \Omega)}{\ue - U} \leq C \varepsilon^d.
	\label{eq:border}
	\end{equation}
\label{pr:border}
\end{proposition}
\vspace{-0.75cm}
\begin{proof}
	The Neumann function $N_U$ of the operator $- \Delta + 3 U^2$ can be written as: 
	\begin{equation}
		N_U(x,y) = \Gamma(x,y) + \tilde{z}(x,y) \quad \forall x,y \in \Omega,\quad x \neq y,
	\label{eq:N}
	\end{equation}
	where $\Gamma$ is the fundamental solution of the operator $-\Delta$, i.e.
	\[
		\Gamma(x,y) = -\frac{1}{2\pi} ln|x-y| 
	\]
	and for every $y \in \Omega$, $z(\cdot,y)$ is the solution of
	\begin{equation}
		\left\{
	\begin{aligned}
		- \Delta_x \tilde{z}(\cdot,y) + 3U^2 \tilde{z}(\cdot,y) &= -3U^2\Gamma(\cdot,y) \qquad &\text{in } \Omega \\
		\partial_{\normal} \tilde{z}(\cdot,y) &= -\partial_{\normal} \Gamma(\cdot,y) \qquad &\text{on } \partial \Omega.
	\end{aligned}
	\right.
	\label{eq:z}
	\end{equation}
	Taking advantage of hypothesis \eqref{eq:sep}, we consider $dist(y,\partial \Omega) \geq d_0$; then, $\Gamma(\cdot,y) \in L^2(\Omega)$ and $\partial_{\normal} \Gamma(\cdot,y)|_{\partial \Omega} \in H^{1/2}(\partial \Omega)$, and by regularity results on elliptic equations (see e.g. \cite{book:evans}, \cite{book:adams}) one may conclude that $z \in H^2(\Omega)$ and $z|_{\partial \Omega} \in H^{3/2}(\partial \Omega)$. In particular, $\norm{L^2(\partial \Omega)}{z|_{\partial \Omega}}$ and $\norm{L^2(\partial \Omega)}{\nabla z|_{\partial \Omega}}$ are bounded by a constant $\tilde{C}_1=\tilde{C_1}(d_0,\Omega)$.
	Moreover, we report the following estimates (see Proposition $4.2$ in \cite{art:bcmp}):
	\begin{equation}
	\begin{aligned}
		\norm{L^{\infty}(\Omega)}{U} &\leq \tilde{C}(\norm{L^p(\Omega)}{f}+\norm{L^p(\Omega)}{f}^3) \leq \tilde{C_2} \\
	  \norm{L^{\infty}(\Omega)}{\nabla U} &\leq \tilde{C}(\norm{L^p(\Omega)}{f}+\norm{L^p(\Omega)}{f}^3) \leq \tilde{C_2} = \tilde{C}_2(\norm{L^p(\Omega)}{f}).
	\end{aligned}
	\label{eq:Uinf}
	\end{equation}
	Hence, from the expansion \eqref{eq:exp2},
	\[
		\begin{aligned}
		& \norm{L^2(\partial \Omega)}{\ue - U}^2  = \int_{\partial \Omega}{|\ue-U|^2} \\
		& \qquad \leq 2(1-k)^2\varepsilon^{2d} \int_{\partial \Omega}{(\mathcal{M} \nabla U(z) \cdot \nabla N_U(z,y))^2 d\sigma} + 2\varepsilon^{2d}\int_{\partial \Omega}{U^6(z) N_U^2(z,y) d\sigma} + o(\varepsilon^{2d}) \\
		&\qquad  \textit{(exploiting \eqref{eq:Uinf})} \quad \leq C \varepsilon^{2d}(\norm{L^2(\partial \Omega)}{\nabla N_U}^2 + \norm{L^2(\partial \Omega)}{N_U}^2) + o(\varepsilon^{2d})  \\
		&\qquad  \leq C \varepsilon^{2d}\left( \int_{\partial \Omega}{|\nabla \Gamma(z,y)|^2} + \int_{\partial \Omega}{|\nabla \tilde{z}(z,y)|^2} + \int_{\partial \Omega}{|\Gamma(z,y)|^2} + \int_{\partial \Omega}{|\tilde{z}(z,y)|^2}\right) + o(\varepsilon^{2d}).
		\end{aligned}
	\]
	Thanks to \eqref{eq:sep}, the regularity of $\Gamma(\cdot,y)$ guarantees that the first and the third boundary integrals in the previous sum are controlled by a constant, whereas the second and the fourth ones are bounded thanks to elliptic regularity, as stated above. Therefore, we can infer that
	\[
	\qquad \norm{L^2(\partial \Omega)}{\ue - U}^2 \leq C \varepsilon^{2d} + o(\varepsilon^{2d}), \quad \text{ where } C = C(d_0,\Omega,f,k,|\partial \Omega|).
	\]
\end{proof}

\section{Representation formula for the topological gradient}
\label{sec:rep}

In this section we describe a topological optimization framework which we can exploit to tackle the solution of the reconstruction problem \eqref{eq:inv}. In particular, let us introduce the following objective functional:
\begin{equation}
	J(\Omegae) = \int_{\partial \Omega}{(\ue-u_{meas})^2 d\sigma},
\label{eq:obj}
\end{equation}
where $\Omegae$ denotes the domain $\Omega$ in which a small inclusion $\omegae$ is inserted, and $\ue$ the corresponding solution of the direct problem \eqref{eq:prob} in $\Omegae$. Hence, we can rephrase the inverse problem as follows: given the boundary datum $u_{meas}$, find $\omegae$ satisfying \eqref{eq:sep} and \eqref{eq:incl2} such that
	\begin{equation} 
		J(\Omegae) \rightarrow \min.
		\label{eq:opt}
	\end{equation}
In order to solve problem \eqref{eq:opt}, we need to describe the variation of the functional $J$ from the unperturbed case (associated to a domain $\Omega$ without inclusions and to the corresponding potential $U$, solution of \eqref{eq:unperturbed}) to the case where an inclusion is present. This calls into play the topological gradient of the functional $J$, although with some differences with the original definition in \cite{art:cea} (see \cite{art:cmm}, \cite{art:ams1}): in the case at hand, indeed, we are perturbing the topology of the domain by inserting inclusions instead of holes.
\par
In particular, hypothesis \eqref{eq:incl2} prescribes that the introduced inclusion is uniquely described by two variables: the position $z$ of the center and the dimension $\varepsilon$. Hence, we can introduce the following simplified notation: hereon we will refer to $J(\Omegae)$ as $j(\varepsilon;z)$. Moreover, we notice that, when $\varepsilon = 0$, the function $j$ does not depend on $z$. 
Hence we define, for the case studied, the \textit{topological gradient} of $J$ evaluated in $\Omega$ as the function $G: \Omega \rightarrow \R$ yielding the following expansion as $\varepsilon \rightarrow 0$:
\begin{equation}
	j(\varepsilon;z)= j(0) + \varepsilon^2 G(z) + o(\varepsilon^2), \qquad z \in \Omega. 
\label{eq:expJ}
\end{equation}
Therefore, at a first-order approximation, the value of $G(z)$ describes the variation of the functional $j$ when introducing a small inclusion of center $z$. This entails that the best strategy to reduce $j$ is to introduce the inclusion in the point where $G$ attains its negative minimum value, provided that this latter exists.
\par
In order to exploit for the sake of reconstruction the topological gradient, it is important to compute it in an alternative way with respect to the one described by the definition \eqref{eq:expJ}; this latter would indeed require the solution of several direct problems for each position $z \in \Omega$ where we want to evaluate the topological gradient $G(z)$. Taking advantage of the preliminary results in section \ref{sec:dir}, we are able to prove a useful representation formula for the topological gradient $G$ in every $z \in \Omega$ which only requires to solve two differential problems.

\begin{theorem}[Representation formula for the topological gradient]
	Under the assumptions of Theorem \ref{th:sviluppo}, the topological gradient $G$ of the funcional $J$ fulfills, for any $z \in \Omega$:
	\begin{equation}
		G(z) = (1-k) \nabla U(z)^T \mathcal{M}(z) \nabla W(z) + U^3(z) W(z),
	\label{eq:TopGrad}
	\end{equation}
	where $W$ is the solution of the following \textit{adjoint problem}:
	\begin{equation} 
	\left\{
	\begin{aligned}
		- \Delta W + 3U^2 W &= 0 \qquad &\text{in } \Omega \\
		\partial_{\normal} W &= U - u_{meas} \qquad &\text{on } \partial \Omega.
	\end{aligned}
	\right.
	\label{eq:adjoint}
	\end{equation}
\label{th:TopGrad}
\end{theorem}
\begin{proof}[Proof]
	Let us denote by
	\[
	\normb{v} = \norm{L^2(\partial \Omega)}{v}, \quad (v,w)_{\partial \Omega} = \int_{\partial \Omega}{v w d\sigma}.
	\]
	Then, we have:
	\[
	\begin{aligned}
	j(\varepsilon;z)-j(0) & = \half \normb{\uez - u_{meas}}^2 - \normb{U - u_{meas}}^2 \\
	& = \half \normb{\uez}^2 - (\uez,u_{meas})_{\partial \Omega} - \half \normb{U}^2 + (U,u_{meas})_{\partial \Omega} \\ 
	& = \half \normb{\uez - U}^2 - \normb{U}^2 + (\uez,U)_{\partial \Omega} - (\uez,u_{meas})_{\partial \Omega} + (U, u_{meas})_{\partial \Omega} \\
	& = \half \normb{\uez - U}^2 + (\uez - U, U - u_{meas})_{\partial \Omega}.
	\end{aligned}
	\]
	Thanks to Proposition \ref{pr:border}, the first term of the last expression can be estimated as follows:
	\[
		\normb{\uez - U}^2 \leq C \varepsilon^{4} = o(\varepsilon^{2}).
	\]
	The second term, exploiting \eqref{eq:exp2}, can be written as:
\[ 
	\begin{aligned}
		(\uez - U, U - u_{meas})_{\partial \Omega} = \ & \varepsilon^2 \int_{\partial \Omega}{(1-k)\nabla U(z)^T\mathcal{M}(z) \nabla N_U(z;y) (U(y) - u_{meas}(y))) d\sigma(y)} \\
			& + \varepsilon^2 \int_{\partial \Omega}{ U^3(z) N_U(z;y) (U(y) - u_{meas}(y)) d\sigma(y)} + o(\varepsilon^2).
	\end{aligned}
\]
Introducing the auxiliary variable $w$, solution of the problem
\begin{equation}
	\left\{
	\begin{aligned}
		- \Delta w + 3U^2 w &= 0 \qquad \text{in } \Omega \\
		\partial_{\normal} w &= h \qquad \text{on } \partial \Omega,
	\end{aligned}
	\right.
	\label{eq:aux}
	\end{equation}
we obtain the following \textit{representation formulae}:
\begin{equation}
		\int_{\partial \Omega}{N_U(z;y) h(y) d\sigma(y)} = w(z),
\label{eq:rappresentazione}
\end{equation}
\begin{equation}
		\int_{\partial \Omega}{\nabla N_U(z;y) h(y) d\sigma(y)} = \nabla w(z).
\label{eq:rappresentazione2}
\end{equation}
Hence, considering $h(y)=U(y)-u_{meas}(y)$ and the related solution $W$ of problem \eqref{eq:aux}, through \eqref{eq:rappresentazione} and \eqref{eq:rappresentazione2}, we finally obtain:
\[
\begin{aligned}
	(\uez - U, U - u_{meas})_{\partial \Omega} &= \int_{\dOO}{(\uez(y) - U(y))h(y)} \\
	&= \varepsilon^2 \int_{\partial \Omega}{(1-k)\nabla U(z)^T\mathcal{M}(z) \nabla N_U(z;y) h(y)) d\sigma(y)} \\
	&\qquad \qquad + \varepsilon^2 \int_{\partial \Omega}{ U^3(z) N_U(z;y) h(y) d\sigma(y)} + o(\varepsilon^2)\\
	&  = \varepsilon^2 \left[(1-k)\nabla U(z)^T \mathcal{M}(z) \nabla W(z) + U^3(z)W(z) \right] + o(\varepsilon^2),
\end{aligned}
\] 
and thus the formula \eqref{eq:TopGrad}. 
\end{proof} 

\begin{remark}
	An extension of the problem discussed so far which is indeed of interest for the sake of the application we have in mind is the reconstruction of inclusions provided that a set of measured data are available only on a portion of the boundary. This approximates the actual procedure of measuring the electrical potential, recovering information by means of a finite number of electrodes.
	Let $\Gamma \subset \partial \Omega$, $|\Gamma| \neq 0$ be the portion of boundary on which $u_{meas}^{\Gamma}$ is known. The results provided so far for the reconstruction problem \eqref{eq:inv} can be also recovered in this case, starting from the definition of the cost functional
\[
	J^{\Gamma}(\Omegae) = \int_{\Gamma}{(\ue-u_{meas}^{\Gamma})^2 d\sigma},
\]
which leads to a similar definition of topological gradient $G$. It is possible to prove that the same representation formula in \eqref{eq:TopGrad} holds in this case, except for the definition of the adjoint state $W$, which is instead given by the following adjoint problem:
\begin{equation} 
	\left\{
	\begin{aligned}
		- \Delta W + 3U^2 W &= 0 \qquad &\text{in } \Omega \\
		\partial_{\normal} W &= (U - u_{meas}^{\Gamma}) \chi_{\Gamma} \qquad &\text{on } \partial \Omega.
	\end{aligned}
	\right.
	\label{eq:adjointGamma}
\end{equation}
\label{rem:Gamma}
\end{remark}

\section{\textit{One-shot} reconstruction algorithm}
\label{sec:algo}

Taking advantage of the assumptions made so far and of the theoretical results that have been proved, we are now ready to set up a topological gradient-based reconstruction algorithm for the inverse problem \eqref{eq:inv}. In particular, we remark that, under the hypothesis \eqref{eq:incl2}, we resctrict ourselves to the identification of the position of the center of a small inclusion of prescribed shape. This can be performed by exploiting the formula \eqref{eq:expJ} as explained before: if the topological gradient $G$ attains its (negative) minimum in $\bar{z} \in \Omega$,  
\[
\begin{aligned}
	G(\bar{z}) <0 \quad &\Rightarrow \quad j(\varepsilon;\bar{z}) < j(0) \\
	G(\bar{z}) \leq G(z) \textit{ $\forall z \in \Omega$} \quad &\Rightarrow \quad j(\varepsilon;\bar{z}) \leq j(\varepsilon;z) \textit{ $\forall z \in \Omega$},
\end{aligned}
\]
which means that the introduction of a small inhomogeneity at $z = \bar{z}$ yields the maximum negative variation of the functional $J$. Finally, thanks to the adjoint approach, we have obtained the representation formula \eqref{eq:TopGrad} for the topological gradient, which allows to compute $G(z)$ by solving two boundary value problems.
\par
The boundary datum, when dealing with a practical application, is derived from a measurement. Instead, for the sake of testing the algorithm, we suppose in the sequel to know \textit{a priori} the exact shape and location of the inclusion and we solve the direct problem \eqref{eq:prob} to compute the corresponding potential on the whole domain, from which we extract the boundary datum $u_{meas}$. 

\subsection{Identification in presence of a single measurement}

According to the strategy proposed in \cite{art:cmm}, a \textit{one-shot} algorithm based on the topological gradient can be implemented (see Algorithm \ref{al:top}). 
\begin{algorithm}
\begin{algorithmic}
	\REQUIRE domain $\Omega$, forcing term $f$, boundary datum $u_{meas}$
	\ENSURE approximated centre of the inclusion, $\bar{z}$
	\STATE  compute $U$ by solving \eqref{eq:unperturbed};
		\STATE compute $W$ by solving \eqref{eq:adjoint};
		\STATE  determine $G$ according to \eqref{eq:TopGrad};
		\STATE  find $\bar{z}$ s.t. $G(\bar{z}) \leq G(z) \quad \forall z \in \Omega$.
\end{algorithmic}
\caption{Reconstruction of a single inclusion of small dimensions}
\label{al:top}
\end{algorithm}
The numerical approximation of problems \eqref{eq:prob}, \eqref{eq:unperturbed} and \eqref{eq:adjoint} is performed through the Galerkin-Finite Element Method. To this purpose, we introduce a discretization $\Tau_h$ of the domain $\Omega$, e.g. made of triangular elements in dimension $d=2$, and define the discrete subspace $V_h = X_h^r \cap V$, where
\[ 
	X_h^r(\Omega) = \{v \in C(\bar{\Omega}): v|_K \in \mathbb{P}_r(K) \ \forall K \in \Tau_h \},
\]
being $\mathbb{P}_r$ the space of polynomials of degree $r$.
\par
When applying the finite element method on problem \eqref{eq:prob}, whose weak formulation is reported in \eqref{eq:weak}, we must tackle the solution of a nonlinear system of equations. Indeed, introducing the operator $S:V\rightarrow V^*$, $\mathcal{S}(u) = T(u) - F$, the discrete approximation of the direct problem \eqref{eq:prob} reads:
	\begin{equation}
		\textit{find $u_h \in V_h$ s.t.} \quad <\mathcal{S}(u_h), v_h>_{V,V^*} = 0 \quad \forall v_h \in V_h.
		\label{eq:discr}
	\end{equation}
By denoting the basis $\{\varphi_i\}_{i=1}^{N_h}$ of $V_h$ (where $N_h = dim(V_h)$) by
\[
	u_h(x) = \sum_{i=1}^{N_h} u_i \phi_i(x), \quad x \in \Omega,
\]
\eqref{eq:discr} can be equivalently rewritten as the following algebraic system:
\begin{equation}
	\begin{aligned}
		\textit{find $\underline{u} \in \R^{N_h}$ s.t.}& \quad \underline{S}(\underline{u}) = 0, \\
		\quad \textit{ being } S_i(\underline{u}) = <\mathcal{S}(u_h),& \varphi_i>_{V,V^*} \textit{ and } (\underline{u})_i = u_i,
	\end{aligned}
		\label{eq:discrS}
	\end{equation}
	which is a nonlinear system (due to nonlinearity of $T$) of $N_h$ equations in $N_h$ unknowns. One of the most common strategies to tackle the nonlinearity is the Newton method, which generates a sequence $\{\underline{u}^{(k)}\}$ to approximate the solution $\underline{u}$ as follows:
	\begin{equation}
	\left\{
	\begin{aligned}
		&\underline{u}^{(0)} \text{ given} \\
		&\underline{u}^{(k+1)} = \underline{u}^{(k)} + \underline{\delta u}^{(k)}, \quad k = 0,1,\ldots,
	\end{aligned}
	\right.
	\label{eq:newtonh}
\end{equation}
where $\underline{\delta u}^{(k)}$ is the solution of the linearized system
\begin{equation}
	J(\underline{u}^{(k)}) \underline{\delta u}^{(k)} = - \underline{S}(\underline{u}^{(k)}),
\label{eq:linAlg}
\end{equation}
and $J(\underline{u}^{(k)})$ is the Jacobian matrix of the vectorial function \underline{S}, evaluated at $\underline{u}^{(k)}$.
The sequence $\{\underline{u}^{(k)}\}$ converges to the solution $\underline{u}$ of \eqref{eq:discrS} if $\underline{u}^{(0)}$ is chosen sufficiently close to $\underline{u}$ (according to the Newton-Kantorovich theorem, see e.g \cite{book:kant2}). We remark that problem \eqref{eq:linAlg} is the algebraic counterpart of the following linear problem: find $\delta u_h^{(k)} \in V_h$ s.t.
\begin{equation}
	<d\mathcal{S}(u_h^{(k)})[\delta u_h^{(k)}], v_h>_{V,V^*} = -< \mathcal{S}(u_h^{(k)}),v_h >_{V,V^*} \quad \forall v_h \in V_h,
\label{eq:linWeak}
\end{equation}
where $d\mathcal{S}(w)[\cdot]: V \rightarrow V^*$ is the Frech\'et derivative of $\mathcal{S}$ evaluated at $w$. Hence, the possibility to invert the matrix $J(\underline{u}^{(k)}) \in \R^{N_h \times N_h}$ is equivalent to the well-posedness of \eqref{eq:linWeak}, for which specific details are given in Appendix B. Through this strategy, it is possible to solve the direct problem \eqref{eq:prob} with the exact inclusion in order to obtain the boundary data, as well as the unperturbed problem \eqref{eq:unperturbed} required by Algorithm \ref{al:top}. Differently, the approximation of the adjoint problem \eqref{eq:adjoint}, which is a linear problem, immediately leads to the solution of a linear algebraic system, for which a well-posedness is guaranteed via the Lax-Milgram lemma.
	\par
Once $U$ and $W$ have been computed, the expression of the topological gradient $G(z)$ of the function $j$ is given by \eqref{eq:TopGrad}, where one has to exploit the \textit{a priori} knowledge on the shape of the inclusion to choose the proper polarization tensor. For example, whilw looking for circular-shaped inclusions, we obtain (see \eqref{eq:circle}):
\begin{equation}
	G(z) = \frac{2(1-k)}{1+k} |D| \nabla U(z) \cdot \nabla W(z) + U^3(z)W(z).
\label{eq:gradcircle}
\end{equation}
	Thanks to the discretization introduced, the approximation of the value of the topological gradient $G$ is known in each node of the triangulation $\mathcal{T}_h$. Hence, the search for its minimum point $\bar{z}$ is performed by a simple inspection between the nodal values of $G$. This of course requires the usage of a sufficiently fine mesh $\Tau_h$; otherwise, one may use any finite-dimensional optimization algorithm but entailing the evaluation of $G$ (and possibly its derivative, namely the Hessian of $j$) in points where the values of $U$ and $W$ have not been computed.

\subsection{Identification in presence of multiple measurements}

The proposed Algorithm \ref{al:top} allows to reconstruct the position of the exact inclusion with a single measurement of the boundary datum. However, it exploits a first-order expansion of the cost functional, and this can affect the precision of the reconstruction, due to the disregarded higher-order terms. In order to overcome this drawback, similarly to the approach proposed in \cite{art:cmm}, it is possible to take advantage of \textit{multiple measurements}. Consider $N^f >1$ different non zero forcing terms $f_i, i = 1,\ldots, N^f$, and suppose to know the respective boundary data $u_{meas,i}$, that is, the solutions of the direct problem \eqref{eq:prob} with the same inclusion $\omegaez$ and the corresponding source term $f_i$. Introduce the cost functional
\[
	J(\Omegae) = \sum_{i=1}^N \alpha_i J_i(\Omegae), \quad \text{where }  J_i(\Omegae) = \int_{\partial \Omega}(\uez - u_{meas,i})^2
\]
and $\{ \alpha_i \}_{i=1}^{N^f}$ is a set of weights such that 
\[
\alpha_i > 0, \textit{  } \sum_{i=1}^{N^f} \alpha_i = 1.
\]
Then, the minimum point $\bar{z}$ of the topological gradient $G(z) = \sum_i \alpha_i G_i(z)$ provides a better approximation of the inclusion's center than the minima $\bar{z}_i$ of each $G_i$, the topological gradient of $J_i$, filtering possible errors induced by the asymptotic analysis carried out on each functional $J_i$.
Hence, we perform a slight variation of Algorithm \ref{al:top}, in the case where multiple observations are available:
\begin{algorithm}
\begin{algorithmic}
\REQUIRE domain $\Omega$, forcing terms $f_i$, boundary data $u_{meas,i}$, $i=1,\ldots,N^f$
\ENSURE approximated center of the inclusion, $\bar{z}$

	\FOR{$i = 1,\ldots,N^f$}
		\STATE compute $U_i$ by solving \eqref{eq:unperturbed} with forcing term $f_i$;
			\STATE compute $W_i$ by solving \eqref{eq:adjoint} with Neumann datum $U_i - u_{meas,i}$;
			\STATE determine $G_i$ according to \eqref{eq:TopGrad};
		\ENDFOR
	\STATE compute $G(z) = \sum_{i=1}^{N^f} \alpha_i G_i(z)$;
	\STATE find $\bar{z}$ s.t. $G(\bar{z}) \leq G(z) \quad \forall z \in \Omega$.
\end{algorithmic}
\caption{Reconstruction of a single inclusion, many measurements}
\label{al:top2}
\end{algorithm}
\par
A possible way to define the weights $\{\alpha_1, \ldots, \alpha_{N^f} \}$ is to take
\begin{equation}
	\alpha_i = \frac{j_i(0)/|\min_{\Omega}G_i|}{\sum_{i = 1}^{N^f} j_i(0)/|\min_{\Omega}G_i|}, \quad i=1,\ldots,N^f,
\label{eq:alphai}
\end{equation}
which entails that the information provided by the topological gradient $G_i$ associated to a large value of the cost functional $j_i(0)$ is considered to carry more significative information than the one associated to a smaller value $G_j$, $ j \neq i$. We remark that this requires the calculation (for each $i=1,\ldots,N^f$) of $j_i(0) = \int_{\Gamma_i}{(u_{meas}^{\Gamma} - U)^2}$, which does not yield a significant computational cost, once the unperturbed problem \eqref{eq:unperturbed} has been solved.

\subsection{Partial measurements}
We describe another alternative to Algorithm \ref{al:top}, related to Remark \ref{rem:Gamma}, which is more interesting for the sake of application. Suppose to have information on the boundary potential on a portion $\Gamma$ of $\partial \Omega$ of the form:
\begin{equation}
	\Gamma = \bigcup_{i=1}^{N^{\Gamma}}\Gamma_i,
\label{eq:Gammai}
\end{equation} 
with $\Gamma_i$ open, connected, $|\Gamma_i| \neq 0$ for all $i = 1, \ldots, N^{\Gamma}$. This configuration can model the presence of $N^{\Gamma}$ different measurement devices on the boundary of the domain, on which we recover information of the potential $u_{meas}^{\Gamma}$. Moreover, as in Algorithm \ref{al:top2}, we set up the optimization of an averaged cost functional
\[
	J(\Omegae) = \sum_{i=1}^{N^\Gamma} \alpha_i J_i(\Omegae), \quad \text{where now} \quad J_i(\Omegae) = \int_{\Gamma_i}{(\uez - u_{meas})^2}
\]
is the cost functional related to the single portion $\Gamma_i$ of the boundary. This yields an alternative reconstruction procedure, involving multiple partial measurements obtained with the same forcing term $f$, as reported in Algorithm \ref{al:top3}.
\begin{algorithm}
\begin{algorithmic}
\REQUIRE domain $\Omega$, forcing term $f$, boundary data $u_{meas}^{\Gamma}$
\ENSURE approximated centre of the inclusion, $\bar{z}$
\STATE compute $U$ by solving \eqref{eq:unperturbed} with forcing term $f$;
	\FOR{$i = 1,\cdots,N^{\Gamma}$}
\STATE compute $W_i$ by solving \eqref{eq:adjoint} with Neumann datum $(U - u_{meas}^{\Gamma}) \chi_{\Gamma_i}$;
			\STATE determine $G_i$ according to \eqref{eq:TopGrad};
		\ENDFOR
	\STATE compute $G(z) = \sum_{i=1}^{N^f} \alpha_i G_i(z)$;
	\STATE find $\bar{z}$ s.t. $G(\bar{z}) \leq G(z) \quad \forall z \in \Omega$.
\end{algorithmic}
\caption{Reconstruction of a single inclusion, partial measurements}
\label{al:top3}
\end{algorithm}
\par
We remark that, since the formula \eqref{eq:TopGrad} for the topological gradient and the adjoint problem \eqref{eq:adjoint} are linear with respect to $W$, using homogeneous weights $\alpha_i = 1/N^{\Gamma}$ would be equivalent to rely on Algorithm \ref{al:top} with boundary data acquired on the whole $\Gamma$. Instead, the choice of weights proposed in \eqref{eq:alphai} allows to assing a better predictive value to the information derived by measurements on the portions $\Gamma_i$ which correspond to larger values of the cost functionals $j_i$. 

\section{Numerical Results}
\label{sec:results}

We now show some numerical results obtained by applying Algorithms \ref{al:top}, \ref{al:top2} and \ref{al:top3} in several benchmark cases. The goal is manifold:
\begin{enumerate}[i)]
	\item first of all (in section \ref{sec:standard}) we verify the effectiveness of the reconstruction, introducing a small inhomogeneity of circular shape in a two-dimensional domain $\Omega$, simulating the associated boundary potential $u_{meas}$ (or $u_{meas,i}$ for $i = 1, \ldots, N^f$, in the case of multiple measurements), and computing the distance between the center of the exact inclusion and the detected one;
	\item in section \ref{sec:non} we assess the feasibility of the algorithms when the shape of the inclusion to detect is unknown, and the reconstruction is performed with the polarization tensor of the circle. Indeed, we exploit hypothesis \eqref{eq:incl2} to assimilate an inclusion of small dimension to a circle, at a first approximation;
	\item in section \ref{sec:partial}, we test the reconstruction of circular inclusion in the case of measures performed on portions of the boundary, according to Algorithm \ref{al:top3}, considering a source term which is significative for the foreseen application;
	\item finally, in section \ref{sec:err}, we verify the stability of the procedure proposed in Algorithm \ref{al:top2} and \ref{al:top3} with respect to the presence of a measurement noise on the datum $u_{meas}$.
\end{enumerate}  
\par
In each experiment, the solution of the differential problems is performed via the Galerkin-Finite Element Method, as explained in section \ref{sec:algo}. In order to properly consider inhomogeneities of small dimensions ($diam(\Omega)/diam(\omegaez) \leq 0.05$), a triangulation of $\Omega$ made by a large number of elements ($\approx 30, 000$) is considered. Indeed, since the position of the inclusion is unknown, it is impossible to perform a local refinement of the mesh (which would increase the quality of the mesh without yielding a large cost due to the greater size of the linear system to solve). However, thanks to the \textit{one-shot} approach, the reconstruction procedure is not expensive at all, and the overall computational time is in general of the order of the minute \footnote{We performed the simulations with a laptop with CPU frequency of $2.10 GHz$, RAM $8GB$}(e.g. when applying Algorithm \ref{al:top2} with $N^f = 2$ sources on a mesh of about $30, 000$ elements, the computational time is about $10''$, whereas Algorithm \ref{al:top3} on a mesh of about $100, 000$ elements with $N^{\Gamma} = 16$ takes almost $100''$).
\par
In the case of multiple observations, we use the source terms proposed in \cite{art:cmm} for the linear problem: $f_1(x,y) = x$, $f_2(x,y) = y$, $f_3(x,y) = xy$, $f_4(x,y) = 0.5(x^2-y^2)$, for all $(x,y) \in \Omega$. This allows to assess the effectiveness of our reconstruction procedure in a benchmark case which is similar to the ones proposed in literature for the linear problem. Similarly e.g. to the results shown in \cite{art:cmm}, also in our case it is not necessary to use more then $N^f=4$ forcing terms: in particular, each simulation is carried out with $N^f = 1, \cdots, 4$ and, if the reconstructed position does not undergo a significant change after the introduction of a new measurement, the procedure is stopped. The weights $\alpha_i$ in the averaged functional are chosen as in \eqref{eq:alphai}. When testing Algorithm \ref{al:top3}, instead, the chosen source term is inspired by the foreseen application.  

\subsection{Circular-shaped inclusion detection}
\label{sec:standard}
We report the numerical results obtained for the detection of the center of small circular inclusions in different positions of the domain $\Omega = B(0,1)$. In Figure \ref{fig:risCerchio} we plot the topological gradient $G(z)$, superimposing its negative minimum (white cross) and the boundary of the exact inclusion (white circle of radius 0.04). The minima detected in all the cases are reported in Table \ref{tab:risCerchio}, where we also compute the Euclidean distance between the reconstructed position and the exact inclusion's center. In the column $N_f$ we specify how many measurements were needed for finding the minima.
\begin{figure}[h!]
		\centering
			\subfloat[Real inclusion: (0,0.1)]{
		    	\includegraphics[width=0.5\textwidth]{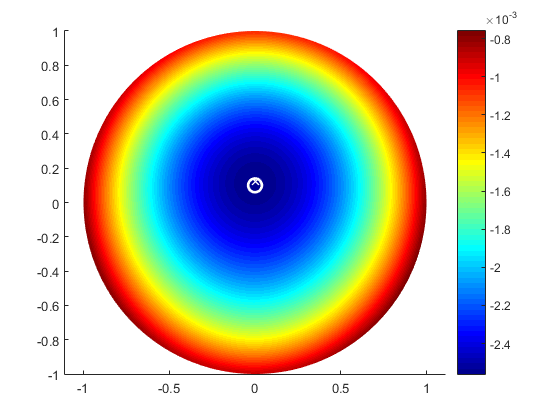}
			}
			\subfloat[Real inclusion: (0.4,0.3)]{
		    	\includegraphics[width=0.5\textwidth]{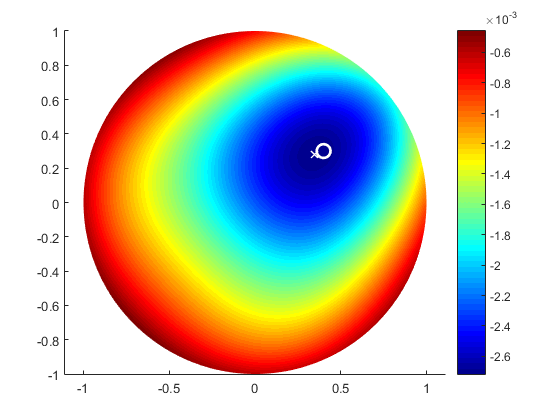}
			}
			\\
			\subfloat[Real inclusion: (-0.6,0)]{
		    	\includegraphics[width=0.5\textwidth]{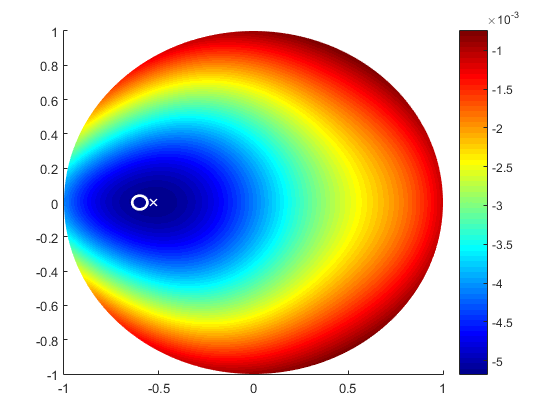}
			}
			\subfloat[Real inclusion: (0.4,-0.5)]{
		    	\includegraphics[width=0.5\textwidth]{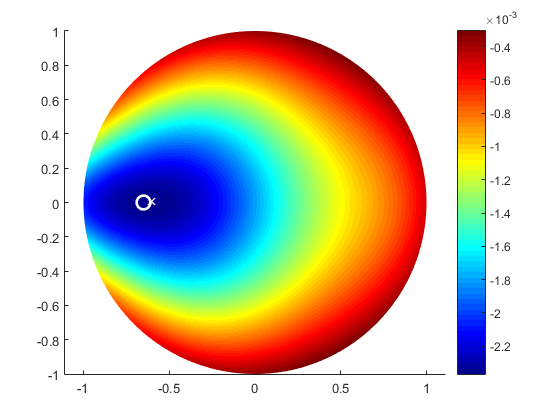}
			}
			\caption[Detection of a circular-shaped inclusion]{Detection of a circular-shaped inclusion}
		\label{fig:risCerchio}
\end{figure}
	
\begin{table}[h!]
\centering
\begin{tabular}[t]{|c|c|c|c|}
\hline
Real inclusion & Detected inclusion & $N_f$ & Error \\
\hline
$(0, 0.1)$ & (0.014,0.106) & 2 & 0.016 \\
$(0.4, 0.3)$ & (0.363, 0.296) & 3 & 0.037  \\
$(-0.65, 0)$ & (-0.603,0.005) & 2 & 0.047 \\
$(0.4, -0.5)$ & (0.431,-0.500) & 3 & 0.031 \\
\hline
\end{tabular}
\caption{Detection of a circular-shaped inclusion: results}
\label{tab:risCerchio}
\end{table}

	
\par
We observe that the algorithm detects the position of the inclusion with an average error of $0.04$ in Euclidean norm, which is comparable to the size of the inclusion itself. Moreover, significant differences can be observed according to the position of the inclusion. Except for the inclusions located very close to the center, it holds that the closer the inclusion to the boundary, the more accurate the reconstruction. However, if the inclusion is eccessively close to boundary, the minimum is detected along the boundary itself, which is of course in contrast with hypothesis \eqref{eq:incl2}.


\subsection{Inclusion of unknown shape}
\label{sec:non}
We can also show that the proposed algorithm is able to detect inclusions of unknown shape. In particular, we are interested to show that formula \eqref{eq:gradcircle}, related to the circular-shaped inclusions, can be successfully applied to detect (at some extent) inclusions whose shape is unknown. In a first case, we reconstruct the center of an inclusion of elliptic shape both with the exact polarization tensor \eqref{eq:ellipse} and with the one related to the circular shape, given \eqref{eq:circle}: see Figure \ref{fig:ellisse1} and Table \ref{tab:ellisse}. Then, we test the identification of inclusions with more involved shapes, for which the polarization tensor is unknown. We report the qualitative results of the detection of an L-shaped inclusion obtained by means of the polarization tensor of the circle (see Figure \ref{fig:Lshape1}).
\begin{figure}[h!]
\vspace{-0.25cm}
		\centering
			\subfloat[Exact polarization tensor]{
		    	\includegraphics[width=0.5\textwidth]{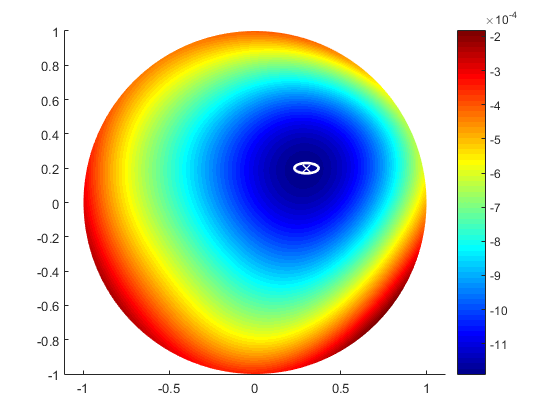}
			}
			\subfloat[Circular-shape tensor]{
		    	\includegraphics[width=0.5\textwidth]{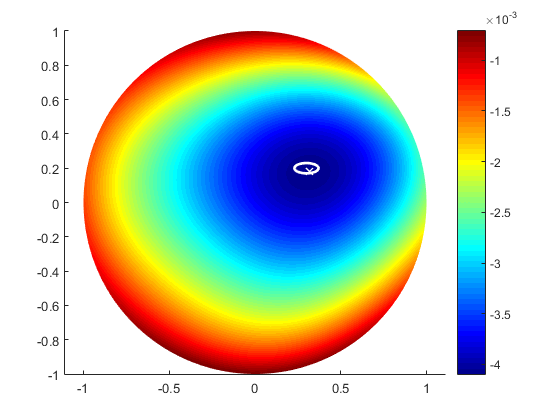}
			}
			\caption[Elliptic-shaped inclusion detection with different tensors]{Elliptic-shaped inclusion detection with different tensors}
		\label{fig:ellisse1}
	\end{figure}
\begin{table}[h!]
\centering
\begin{tabular}[t]{|c|c|c|c|c|}
\hline
$\mathcal{M}$ & Real inclusion center & x and y real semi-axis & Detected center & Error \\
\hline
Ellipse & (0.3, 0.2) & (0.07, 0.03) & (0.302, 0.196) & 0.005 \\
Circle & (0.3, 0.2) & (0.07, 0.03) & (0.320, 0.181) & 0.028 \\
Ellipse & (0.5, 0) & (0.04, 0.02) & (0.487, -0.013) & 0.018 \\
Circle & (0.5, 0) & (0.04, 0.02) & (0.549, 0.009) & 0.050 \\
\hline
\end{tabular}
\caption{Elliptic-shaped inclusion detection with different tensors: results}
\label{tab:ellisse}
\end{table}
\begin{figure}[h!]
		\centering
			\subfloat[Inclusion]{
		    	\includegraphics[width=0.5\textwidth]{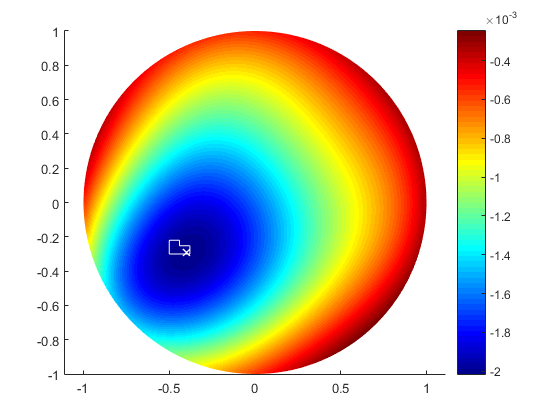}
			}
			\subfloat[Topological gradient]{
		    	\includegraphics[width=0.5\textwidth]{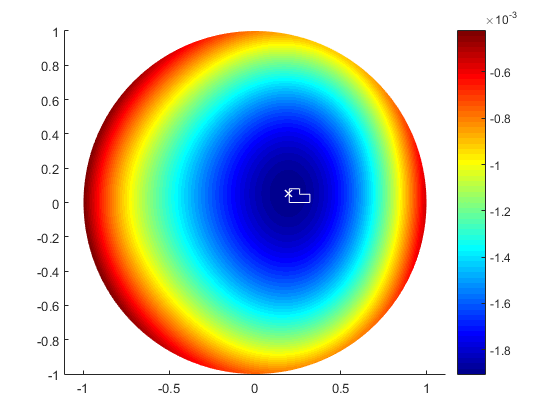}
			}
			\caption[L-shaped inclusion detection using the tensor corresponding to the circular shape]{L-shaped inclusion detection using the tensor corresponding to the circular shape}
		\label{fig:Lshape1}
	\end{figure}

In the case of inclusions of elliptic shape, we remark that the reconstruction error using the polarization tensor of the circle is actually higher, but comparable to the one with the correct tensor. Hence, the proposed \textit{one-shot} algorithm can be used when dealing with the reconstruction of inclusions of unknown shape and small dimensions, providing a first approximation of the center by using the topological gradient associated to the polarization tensor of the circle. This can consist in an initial guess for an iterative scheme, based on the level-set tecnique or on the evaluation of the shape gradient of the functional $J$ if a shape optimization procedure is exploited for the sake of the complete reconstruction of the geometry of the inclusion.

\subsection{Partial measurements}
\label{sec:partial}

In this section, we test Algorithm \ref{al:top3} for the reconstruction of a small circular inhomogeneity in the domain $\Omega = B(0,1)$ using measurements of the potential on a portion $\Gamma$ of the boundary. In particular, we consider a source term $f(x,y) = 1-exp(-r_s^2/((x-x_s)^2+(y-y_s)^2))$, which attains its maximum value in $(x_s,y_s) \in \Omega$ and exponentially decays outside a circular neighborhood of radius $r_s$, approximating the electrical stimulus originated in a specific region. Moreover, the region $\Gamma$ is of the form prescribed by \eqref{eq:Gammai}, where $\Gamma_i$ are equivalent arcs of lenght $2\pi \ell$.
\par
We report some results of the reconstruction algorithm in the case where the exact inclusion has center $(0.5,0.4)$, the forcing stimulus is centered in (0,0) with radius $r_s = 0.3$, $\ell = 1/48$ and we consider different numbers of portions $N^{\Gamma}$: see Table \ref{tab:part} for the quantitative results and Figure \ref{fig:part}, where we marked $\Gamma$ with a thick black line.
\begin{table}[h!]
\centering
\begin{tabular}[t]{|c|c|c|}
\hline
$N^{\Gamma}$ & Detected inclusion &  Error \\
\hline
 8 & (0.629, 0.530) & 0.183 \\
12 & (0.482, 0.346) & 0.057 \\
16 & (0.508, 0.423) & 0.025 \\
24 & (0.489,-0.398) & 0.011 \\
\hline
\end{tabular}
\caption{Reconstruction with partial measurements: results}
\label{tab:part}
\end{table}

\begin{figure}[h!]
		\centering
			\subfloat[$N^{\Gamma} = 8$]{
		    	\includegraphics[width=0.5\textwidth]{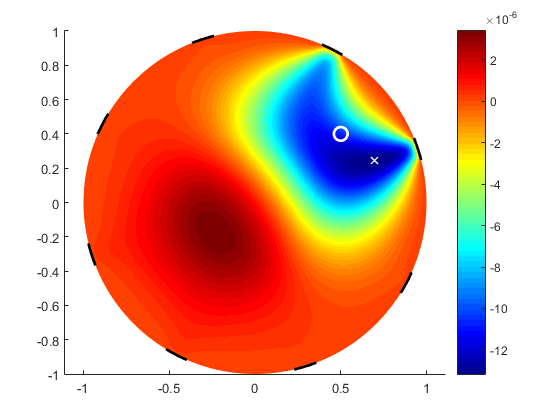}
			}
			\subfloat[$N^{\Gamma} = 12$]{
		    	\includegraphics[width=0.5\textwidth]{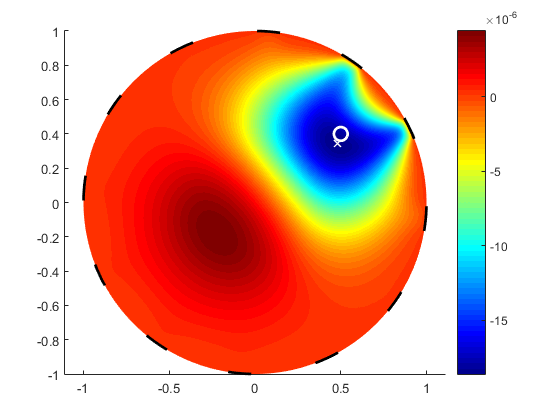}
			} 
			\\
			\subfloat[$N^{\Gamma} = 16$]{
		    	\includegraphics[width=0.5\textwidth]{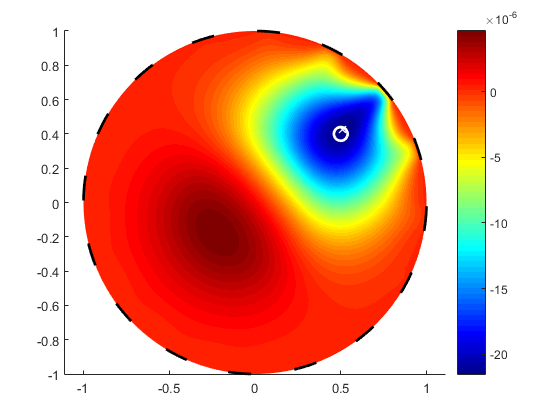}
			}
			\subfloat[$N^{\Gamma} = 24$]{
		    	\includegraphics[width=0.5\textwidth]{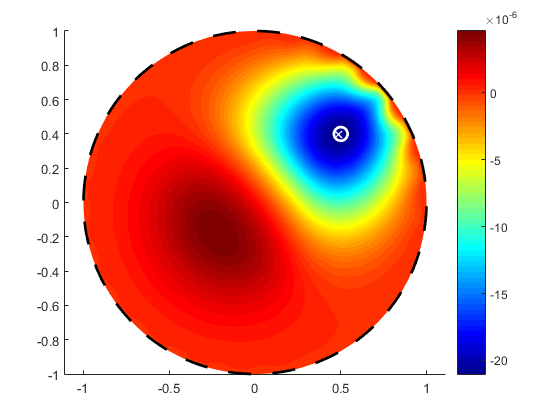}
			}
		\caption{Reconstruction with partial measurements: results}
		\label{fig:part}
	\end{figure}

\subsection{Effect of experimental noise}
\label{sec:err}
In this last subsection, we show the stability of Algorithms \ref{al:top2} and \ref{al:top3} with respect to possible experimental or measurement noise on the boundary data. We perturb the value of the exact solution computed on the boundary up to a fixed percentage $p$ ($\tilde{u}_{meas}(x) = u_{meas}(x)(1-p/2 + rand(x)p)$, where $rand(x)$ is a random number between $0$ and $1$ for each $x \in \Omega$), assessing the performances of the reconstruction procedures. Some results in the case of the reconstruction of circular-shaped inclusions with multiple measurements are reported in Figure \ref{fig:errStat} and in Table \ref{tab:errStat}. We conducted the simulation 100 times with different realizations of the random experimental noise, reporting the average error obtained; the cases where the inclusion was detected on the boundary (and thus the reconstruction fails) are not taken into account, but are reported in Table \ref{tab:errStat} as ``failure'' cases. In Table \ref{tab:errStat2} and in Figure \ref{fig:errStat2}, instead, we report the results obtained in the case of partial measurements affected by noise, in the case of $N^{\Gamma} = 12, 16, 24$ portions of the boundary of length $2\pi \ell$, $\ell = 1/48$. Each simulation was conducted 20 times with different random errors; the average results are then reported.
\begin{table}[h!]
\centering
\begin{tabular}[t]{|c|c|c|c|}
\hline
Percentage & Real inclusion's center & Failure & Mean error \\
\hline
1\%  & (0.2,-0.2) & 0\% & 0.026\\
2\%  & (0.2,-0.2) & 0\% & 0.034\\
5\%  & (0.2,-0.2) & 0\% & 0.082\\
10\% & (0.2,-0.2) & 33\% & 0.212\\
\hline
\end{tabular}
\caption{Results under experimental errors: multiple measurements}
\label{tab:errStat}
\end{table}

\begin{figure}[h!]
\vspace{-0.25cm}
		\centering
			\subfloat[Error: 1\%]{
		    	\includegraphics[width=0.5\textwidth]{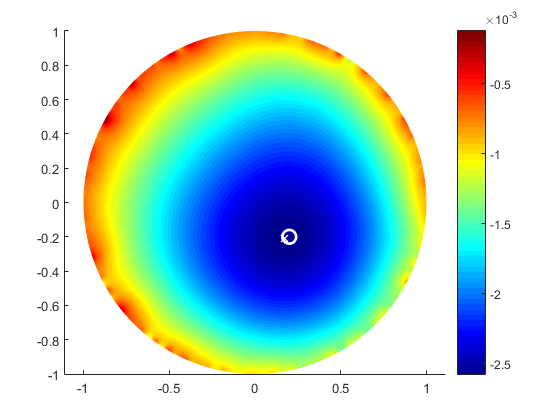}
			}
			\subfloat[Error: 2\%]{
		    	\includegraphics[width=0.5\textwidth]{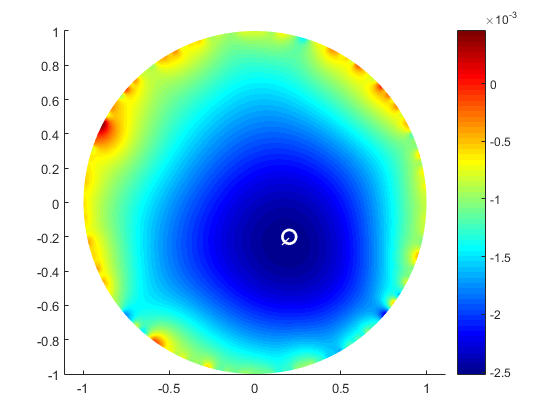}
			}
			\\
			\subfloat[Error: 5\%]{
		    	\includegraphics[width=0.5\textwidth]{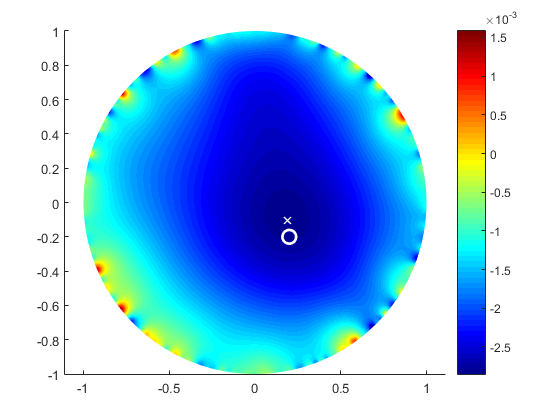}
			}
			\subfloat[Error: 10\%]{
		    	\includegraphics[width=0.5\textwidth]{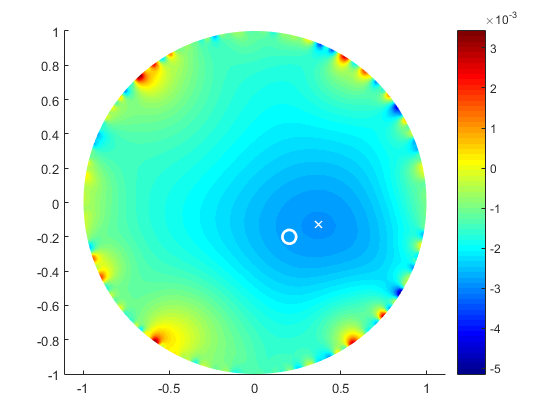}
			}
			\caption[Results under experimental errors: multiple measurements]{Results under experimental errors: multiple measurements}
		\label{fig:errStat}
	\end{figure}
	
	\begin{table}[h!]
\centering
\begin{tabular}[t]{|c|c|c|c|}
\hline
 & $N^{\Gamma} = 12$ & $N^{\Gamma} = 16$ & $N^{\Gamma} = 24$ \\
\hline
$p = 1\%$  	& 0.087 & 0.030 & 0.021 \\
$p = 2\%$   & 0.103 & 0.087 & 0.040 \\
$p = 5\%$   & 0.254 & 0.170 & 0.138 \\
\hline
\end{tabular}
\caption{Results under experimental errors: partial measurements}
\label{tab:errStat2}
\end{table}

\begin{figure}[h!]
		\centering
			\subfloat[$p$ = 1\%, $N^{\Gamma}$ = 12]{
		    	\includegraphics[width=0.32\textwidth]{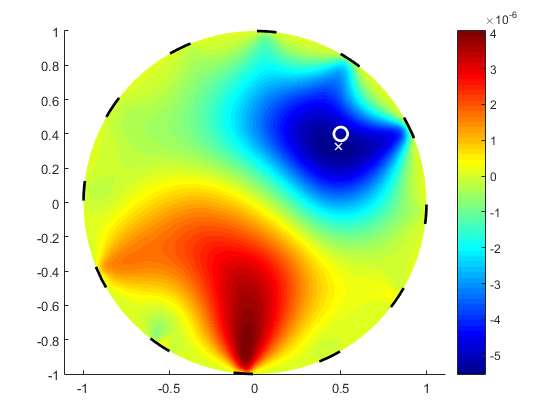}
			}
			\subfloat[$p$ = 1\%, $N^{\Gamma}$ = 16]{
		    	\includegraphics[width=0.32\textwidth]{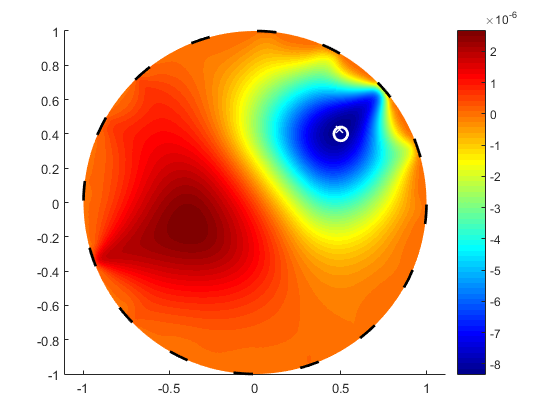}
			}
			\subfloat[$p$ = 1\%, $N^{\Gamma}$ = 24]{
		    	\includegraphics[width=0.32\textwidth]{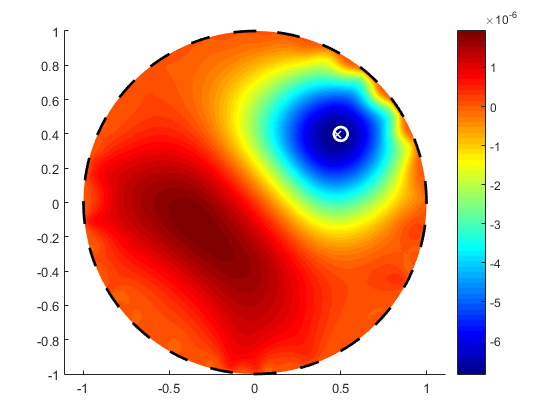}
			}
			\\\subfloat[$p$ = 2\%, $N^{\Gamma}$ = 12]{
		    	\includegraphics[width=0.32\textwidth]{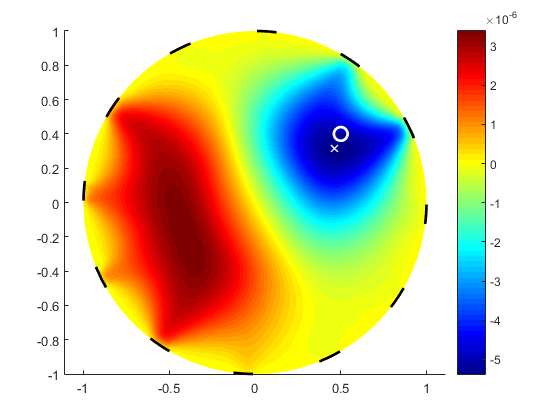}
			}
			\subfloat[$p$ = 2\%, $N^{\Gamma}$ = 16]{
		    	\includegraphics[width=0.32\textwidth]{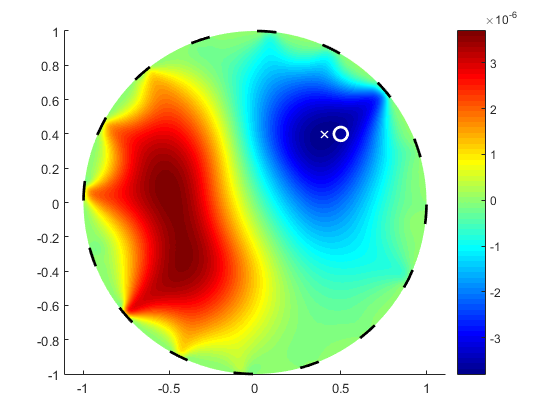}
			}
			\subfloat[$p$ = 2\%, $N^{\Gamma}$ = 24]{
		    	\includegraphics[width=0.32\textwidth]{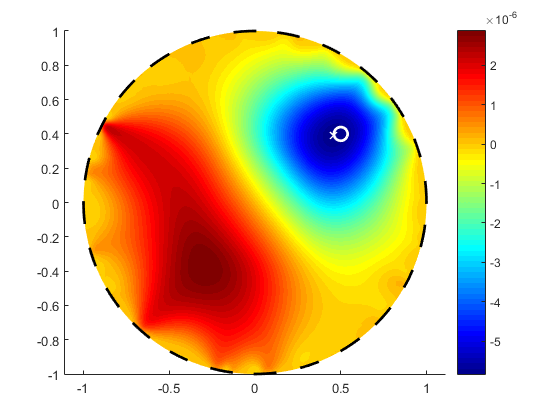}
			}
			\\ \subfloat[$p$ = 5\%, $N^{\Gamma}$ = 12]{
		    	\includegraphics[width=0.32\textwidth]{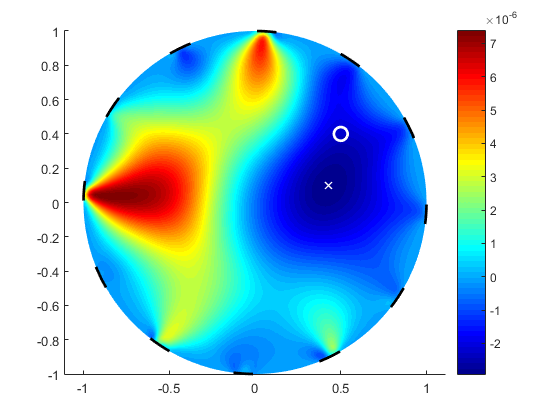}
			}
			\subfloat[$p$ = 5\%, $N^{\Gamma}$ = 16]{
		    	\includegraphics[width=0.32\textwidth]{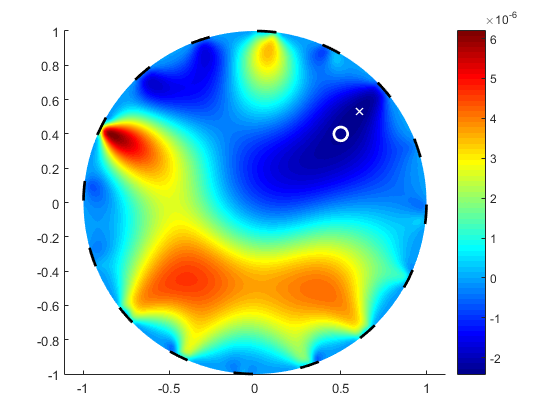}
			}
			\subfloat[$p$ = 5\%, $N^{\Gamma}$ = 24]{
		    	\includegraphics[width=0.32\textwidth]{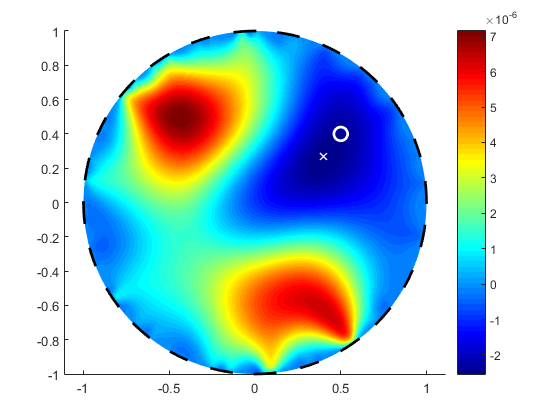}
			}
			\caption[Results under experimental errors: partial measurements]{Results under experimental errors: partial measurements}
		\label{fig:errStat2}
	\end{figure}
\par
We point in the first case, i.e. the reconstruction with many measurements, that the detected position is stable under small perturbations of the data (namely, the error in reconstruction grows almost linearly with respect to the experimental noise), but there exists a threshold value (e.g., in the first case, below 10\%) above which the information provided by the topological gradient is too noisy to be meaningful for the sake of reconstruction. The same happens in the second case, for each number $N^{\Gamma}$ of portions $\Gamma_i$ considered. This suggests the possibility to prove a local stability result for the reconstruction problem in the same way as shown in \cite{art:cfmv} for the (linear) inverse conductivity problem in the case of inclusions of small dimensions: further work is ongoing in this respect.

\section*{Conclusions}
In this work, a new procedure for a nonlinear reconstruction problem has been described, under the assumption that the inclusion to be detected has small dimensions and fixed shape. An equivalent optimization problem has been introduced, which is solved by means of the topological gradient, allowing to define a \textit{one-shot} reconstruction algorithm. This procedure requires the solution of two boundary value problems, the unperturbed one and the adjoint problem. Some extensions of the algorithm have been proposed in order to improve its accuracy, $(i)$ taking advantage of more than one measurements, or $(ii)$ properly exploiting the information coming from different portions of the boundary. These procedures have been tested for the reconstruction of inclusions of circular shape and of \textit{a priori} unknown shape, assessing the accuracy and the stability of the solution when the data are affected by possible experimental noise.
\par
The algorithms have proved to be effective in the reconstruction, yielding the exact position of the inclusion up to an error of $5\%$ by considering at most $N^f = 4$ or $N^{\Gamma} = 24$ measurements. The analysis presented in this work suggests several extensions. In particular, the possibility to have a first reasonable guess of the position of an inclusion of unknown shape is a first, encouraging step towards the coupling of this strategy with an iterative scheme, based e.g. on the shape gradient of a suitable cost function, in order to fully reconstruct the geometry of the inhomogeneity. Moreover, the robustness with respect to the statistical errors suggests the stability analysis for the reconstructed position with respect to a small perturbation of the boundary data. This work can also pave the way to account for several extensions in order to tackle the complexity of the physical model we are interested in for the sake of application in cardiac electrophysiology.

\newpage
%
%

\section*{Appendix A. Details on the forcing terms}
\label{app:f}

Conforming to the hypothesis introduced in \cite{art:bcmp}, the choice of the forcing term $f$ must satisfy the condition: 
\begin{equation}
	\exists m > 0 \text{ s.t. } f(x)\geq m \text{ } \forall x \in \Omega.
\label{eq:m}
\end{equation} 
This restriction can be avoided, as it is possible to weaken \eqref{eq:m} by only assuming that:
\begin{equation}
 f \textit{ does not identically vanish, i.e.} \quad \norm{L^p(\Omega)}{f} \neq 0
\label{eq:vanish}
\end{equation} 
Indeed, in the proof delivered in \cite{art:bcmp}, that hypothesis is only needed for the preliminary estimate proved in Theorem 4.2, namely 
\begin{equation}
	\norm{H^1(\Omega)}{\ue-U} \leq C|\omegae|^{1/2}.
\label{eq:bcmp_estimate}
\end{equation}
In particular, hypothesis \eqref{eq:m} is required in order to prove that, in order to obtain \eqref{eq:bcmp_estimate}, the following estimate from below holds: 
\[
	\exists C = C(\abs{\Omega \setminus \omegae},m) > 0 \textit{ s.t } \int_{\Omega \setminus \omegae}{q_{\varepsilon}} > C, \quad \textit{ where } q_{\varepsilon} = U^2 + U \ue + \ue^2.
\]   
Hence, we can substitute it as follow:
	\textit{if $f \in L^p(\Omega)$ satisfies \eqref{eq:vanish}, then $\exists C = C(\abs{\Omega \setminus \omegae},f) > 0$ s.t. $\int_{\Omega \setminus \omegae}{q_{\varepsilon}} > C$}.
Indeed, if $f$ satisfies \eqref{eq:vanish}, then also $U$ cannot identically vanish, otherwise it could not solve \eqref{eq:unperturbed}. Hence, denoting $M = \norm{\infty}{U}$, we can ensure that $M > 0$. Consider
\[
\Omega_0 = \{x \in \Omega: |U(x)| \leq M/2\} \text{ and } \Omega_1 = \{x \in \Omega: |U(x)|> M/2\}:
\]
as $U$ is continuous in $\Omega$ (see Proposition 4.2 in \cite{art:bcmp}, indipendent of hypothesis \eqref{eq:m}), we conclude that $|\Omega_1| > 0 $. We introduce $\tilde{U}$ defined as follows:
\[
	\tilde{U}(x) = \left\{ 
	\begin{aligned}
		M/2 \quad &x \in \Omega_1 \\
		0 \quad &x \in \Omega_0.
	\end{aligned}
	\right.
\]
By definition, $U^2(x) \geq \tilde{U}^2(x) \quad \forall x \in \Omega$; hence, we obtain:
\[
	\int_{\Omega \setminus \omegae}{q_{\varepsilon}} \geq \int_{\Omega \setminus \omegae}{\frac{3}{4} U^2} \geq\int_{\Omega \setminus \omegae}{\frac{3}{4} \tilde{U}^2} \geq \tilde{C} M^2 |\Omega_1| = C > 0.
\]
We remark that hypothesis \eqref{eq:m} allows to write an estimate for the quantity $q_{\varepsilon}$ (and therefore for $\norm{H^1(\Omega)}{\ue-U}$) which is independent of the choice of the forcing term $f$, whereas the weakened one, \eqref{eq:vanish}, entails an estimate which depends on $M$ and $\Omega_1$ and ultimately on the choice of $f$. This allows to use the main theoretical results (Theorem \ref{th:sviluppo} and consequences) in the proposed weaker hypothesis, and does not compromise the effectiveness of such estimate in the case of our application.

\section*{Appendix B. On the well-posedness of the linearization involved in the numerical approximation of the direct problem}

\label{app:New}
As explained in section \ref{sec:algo}, the numerical approximation of the direct problem requires to exploit the iterative Newton algorithm for nonlinear systems. The linearized problem \eqref{eq:linAlg}, which we have to solve at each step, is the algebraic formulation of the problem \eqref{eq:linWeak}, which explicity reads: find $u_h \in V_h$ such that
\begin{equation}
	\begin{aligned}
		\int_{\Omega}{k(x) \nabla \delta u_h \cdot \nabla v_h }  &+ \int_{\Omega \setminus \omega}{3 \left(u_h^{(k)} \right)^2 \delta u_h v_h} =  \\
		= \int_{\Omega}{f v_h} -& \int_{\Omega}{k \nabla u_h^{(k)}  \cdot \nabla v } - \int_{\Omega \setminus \omega}{\left(u_h^{(k)}\right)^3 v_h} \quad \forall v_h \in V_h.
	\end{aligned}
\label{eq:newtonhw}
\end{equation}

Such a problem is indeed well-posed, according to the following result:
\begin{proposition}
	If $u_h^{(0)} \in V_h \cap L^{\infty}(\Omega)$, problem \eqref{eq:newtonhw} admits an unique solution $u_h^{(k)}$ in $V_h \subset V=H^1(\Omega)$ for every $k$. Moreover, 
\[
	\exists C_k > 0 \text{ s.t. } \norm{H^1(\Omega)}{u_h^{(k)}} \leq C_k \left(\norm{H^{-1}(\Omega)}{f} + \norm{H^1(\Omega)}{u_h^{(0)}}^3 \right).
\]
\label{th:linearizedh}
\end{proposition}
\begin{proof}[Proof]
Consider the first iteration: for a fixed initial point $u_h^{(0)}$ in $V_h \cap L^{\infty}(\Omega)$, the linear operator $-s(u_h^{(0)},\cdot) = -< \mathcal{S}(u_h^{(0)}), \cdot >_{V,V^*}$ and the bilinear form $ds[u_h^{(0)}](\cdot, \cdot) = <d_{u_h^{(0)}}\mathcal{S}\cdot, \cdot>_{V,V^*} $ are continuous: for all $u_h, v_h \in V_h$, it holds:
\[
\begin{aligned}
\abs{ds[u_h^{(0)}](u_h,v_h)} &\leq k_{out} \norm{L^2(\Omega)}{\nabla u_h} \norm{L^2(\Omega)}{\nabla v_h} + 3 \norm{L^{\infty}(\Omega)}{u^{(0)}}^2 \norm{L^2(\Omega)}{u_h} \norm{L^2(\Omega)}{v_h} \\
 &\leq max \left\{k_{out}, 3 \norm{L^{\infty}(\Omega)}{u_h^{(0)}}^2 \right\}\norm{H^1(\Omega)}{u_h} \norm{H^1(\Omega)}{v_h}; \\
\abs{s(u_h^{(0)},v_h)} & \leq \norm{L^2(\Omega)}{f} \norm{L^2(\Omega)}{v_h} + k_{out} \norm{L^2(\Omega)}{\nabla v_h}\norm{L^2(\Omega)}{\nabla u_h^{(0)}} + \norm{L^6(\Omega)}{u_h^{(0)}}^3\norm{L^2(\Omega)}{v_h} \\
&\leq \left(  \norm{L^2(\Omega)}{f} + max \left\{k_{out}, C_{Sob}^2 \norm{H^1(\Omega)}{u_h^{(0)}}^3 \right\} \right)\norm{H^1(\Omega)}{v_h} .
\end{aligned}
\]
Nevertheless, the bilinear form is not coercive in $V_h$, indeed:
\[
ds[u_h^{(0)}](u_h,u_h) = \int_{\Omega}{k \nabla u_h \cdot \nabla u_h }  + \int_{\Omega}{3 \chi_{\Omega \setminus \omega} \left(u_h^{(k)} \right)^2 u_h^2}
\]	
and a lower bound of the latter quantity in terms of the $H^1$-norm of $u_h$ cannot be obtained because of the presence of the indicator function over $\Omega \setminus \omega$ in the reaction term. The weak coercivity is instead guaranteed, with constant $k_{in}>0$:
\[
\begin{aligned}
ds[u_h^{(0)}](u_h,u_h) + k_{in} \norm{L^2(\Omega)}{u_h} &=  \int_{\Omega}{k \nabla u_h \cdot \nabla u_h }  + \int_{\Omega}{3 \chi_{\Omega \setminus \omega} \left(u_h^{(0)} \right)^2 u_h^2} + \int_{\Omega}{k_{in} u_h^2} \\
& \geq \int_{\omega}{k_{in} \nabla u_h \cdot \nabla u_h } + \int_{\Omega \setminus \omega}{k_{out} \nabla u_h \cdot \nabla u_h } + \int_{\Omega}{k_{in} u_h^2} \\
& \geq k_{in} \norm{L^2(\Omega)}{\nabla u_h}^2 + k_{in} \norm{L^2(\Omega)}{u_h}^2 = k_{in} \norm{H^1(\Omega)}{u_h}.
\end{aligned}
\]	
Hence, it is possible to apply the Ne\c cas theorem (or the Fredholm Alternative), see e.g. \cite{book:evans}, Chapter 6, which entails the well-posedness of the problem \eqref{eq:newtonhw} for $k=0$ only if the \textit{homogeneous problem} has an unique solution, i.e.:
\begin{equation}
ds[u_h^{(0)}](u_h,v_h) = 0 \quad \forall v_h \in V_h \quad \Leftrightarrow \quad u_h = 0.
\label{eq:omogeneo}
\end{equation}
To prove it, consider that, if $w \in V_h \subset V = H^1(\Omega)$ solves \eqref{eq:omogeneo}, then it also satisfies:
\[
	\int_{\Omega}{k \nabla w \cdot \nabla v_h} + \int_{\Omega \setminus \omega}{3\left(u_h^{(0)}\right)^2wv_h} = 0 \quad \forall v_h \in V_h.
\] 
Hence, with $v_h = w$,
\begin{equation}
	\nabla w = 0 \text{  in } \Omega, \quad w = 0 \text{  in } \Omega \setminus \omega.
	\label{eq:w}
\end{equation}
In order to conclude that $w = 0$ in $\Omega$, it is necessary to prove some extra regularity conditions on the solution $w$: in particular, it is possible to show that $w \in C^{0,\alpha}(\Omega)$. Indeed, the term $\chi_{\Omega \setminus \omega} 3\left(u_h^{(0)}\right)^2w$ is bounded in $L^p(\Omega)$ for all $p \geq 1$, because of the choice of the starting solution. Moreover, $w$ satisfies
\[
	-\diverg(k \nabla w) = -\chi_{\Omega \setminus \omega} 3\left(u_h^{(0)}\right)^2w,
\]
and exploiting Theorem 8.24 in \cite{book:gt} we can prove the interior estimate: 
\begin{equation}
\forall \Omega' \subset \subset \Omega, \quad \exists C, \alpha >0: \quad \norm{C^{0,\alpha}(\bar{\Omega'})}{w} \leq C \left( \norm{L^2(\Omega')}{w} + \norm{L^p(\Omega)}{3\left(u_h^{(0)}\right)^2w} \right).
\label{eq:gt}
\end{equation}
Hence, by fixing $\Omega ' \supset \omega$ and considering that $k(x) = k_{out}$ is constant on $\Omega \setminus \Omega'$, under mild regularity hypothesis on the domain $\Omega$, one may recover the Holder-continuity of $w$ on the whole $\bar{\Omega}$, entailing the uniqueness of the solution of \eqref{eq:omogeneo} and thus the well-posedness of \eqref{eq:newtonhw} for $k=0$.
\\The stability estimate is guaranteed by Ne\c cas' theorem, yielding
\[
	\norm{H^1(\Omega)}{u_h^{(1)}} \leq \frac{1}{k_{in}}\norm{H^{-1}}{g(u_h^{(0)},\cdot)} \leq C_k (\norm{L^2(\Omega)}{f} + \norm{H^1(\Omega)}{u_h^{(0)}}^3).
\]
Hence, the solution $\delta u_h^{(0)}$ of \eqref{eq:newtonhw} with $k = 0$ exists and is unique in $V_h$, and with a procedure similar to the one used on the homogeneous problem, one may prove that $\delta u_h^{(0)} \in C^{0,\alpha} \supset L^{\infty}(\Omega)$. Moreover, also $u_h^{(1)} \in H^1(\Omega) \cap L^{\infty}(\Omega)$, and this can be iterated to prove the thesis on each $k>0$, by induction.
\end{proof}
We have therefore set the numerical strategy for the approximate solution of problem \eqref{eq:weak}: the well-posedness of the algebraic problems \eqref{eq:linAlg} to be solved at each step is entailed by the latter proposition, whereas the convergence of the sequence $\{ u_h \}_{h>0}$ is guaranteed by the Kantorovich theorem (see \cite{book:kant2}), which exploits the Lipschitz-continuity of the functional $\mathcal{S}(u) = T(u) - F$.

\addcontentsline{toc}{section}{\refname}
\printbibliography
\end{document}